\documentclass{amsart}
\usepackage{amsmath}
\usepackage{amssymb}
\usepackage{amsfonts}
\usepackage{graphicx}

\makeatletter \@namedef{subjclassname@2010}{  \textup{2010}
Mathematics Subject Classification}
\theoremstyle{plain}
\newtheorem{theorem}{Theorem}[section]
\newtheorem{corollary}[theorem]{Corollary}
\newtheorem{proposition}[theorem]{Proposition}
\newtheorem{lemma}[theorem]{Lemma}
\theoremstyle{remark}

\theoremstyle{definition}
\theoremstyle{example}

\numberwithin{equation}{section}
\input{tcilatex}

\begin{document}
\title[POLYNOMIAL GROWTH CONDITIONS ]{OPERATORS WHOSE CONJUGATION\ ORBITS\
SATISFY\ POLYNOMIAL\ GROWTH CONDITIONS }
\author{HEYBETKULU MUSTAFAYEV}
\address{Van Yuzuncu Yil University, Faculty of Science, Department of
Mathematics, VAN-TURKEY}
\email{hsmustafayev@yahoo.com}
\subjclass[2010]{ 47A10, 47A11, 30D20.}
\keywords{Operator, growth condition, (local) spectrum, Beurling algebra,
decomposability.}

\begin{abstract}
Let $A$ be a bounded linear operator on a complex Banach space $X.$ For a
given $\alpha \geq 0,$ we consider the class $\mathcal{D}_{A}^{\alpha
}\left( 
\mathbb{R}
\right) $ of all bounded linear operators $T$ on $X$ for which there exists
a constant $C_{T}>0$, such that 
\begin{equation*}
\left\Vert e^{tA}Te^{-tA}\right\Vert \leq C_{T}\left( 1+\left\vert
t\right\vert \right) ^{\alpha },\text{ \ }\forall t\in 
\mathbb{R}
.
\end{equation*}%
We present complete description of the class $\mathcal{D}_{A}^{\alpha
}\left( 
\mathbb{R}
\right) $ in the case when the spectrum of $A$ consists of one point. These
results are linked to the decomposability of $A.$ Some estimates for the
norm of the commutator $AT-TA$ are obtained in the case $0\leq \alpha <1.$
\end{abstract}

\maketitle

\section{Introduction}

Let $H$ be an infinite dimensional separable Hilbert space and let $B\left(
H\right) $ be the algebra of all bounded linear operators on $H.$ A family $%
\left\{ E\right\} $ of subspaces of $H$ is called a \textit{nest} if it is
totally ordered by inclusion. Given a nest $\left\{ E\right\} ,$ Ringrose 
\cite{16} introduced the concept of the associated\textit{\ nest algebra} Alg%
$\left\{ E\right\} $ defined by 
\begin{equation*}
\text{Alg}\left\{ E\right\} =\left\{ T\in B\left( H\right) :TE\subseteq E%
\text{, }\forall E\in \left\{ E\right\} \right\} .
\end{equation*}%
In \cite{10}, Leobl and Muhly show that every nest algebra is the algebra of
all analytic operators with respect to the one-parameter representation $%
T\rightarrow e^{-itA}Te^{itA}$ of inner $\ast -$automorphisms of $B\left(
H\right) ,$ where $A$ is a self-adjoint operator on $H.$

For an invertible operator $A$ on $H$, Deddens \cite{4} introduced the set 
\begin{equation*}
\mathcal{D}_{A}:=\left\{ T\in B\left( H\right) :\sup_{n\geq 0}\left\Vert
A^{n}TA^{-n}\right\Vert <\infty \right\} .
\end{equation*}%
Notice that $\mathcal{D}_{A}$ is a normed (not necessarily closed) algebra
with identity and contains the commutant $\left\{ A\right\} ^{\prime }$ of $%
A.$ In the same paper, Deddens shows that if $A\geq 0,$ then $\mathcal{D}%
_{A} $ coincides with the nest algebra associated with the nest of spectral
subspaces of $A.$ This gives a new and convenient characterization of nest
algebras. In \cite{4}, Deddens conjectured that the equality $\mathcal{D}%
_{A}=\left\{ A\right\} ^{\prime }$ holds if the spectrum of $A$ is reduced
to $\left\{ 1\right\} .$ In \cite{17}, Roth gave a negative answer to
Deddens conjecture. He shows existence of a quasinilpotent operator $A$ for
which $\mathcal{D}_{I+A}\neq \left\{ I+A\right\} ^{\prime }.$

Let $X$ be a complex Banach space and let $B\left( X\right) $ be the algebra
of all bounded linear operators on $X.$ In \cite{18}, Williams proved that
if the spectrum of an invertible operator $A\in B\left( X\right) $ is
reduced to $\left\{ 1\right\} $ and $\sup_{n\in 
\mathbb{Z}
}\left\Vert A^{n}TA^{-n}\right\Vert <\infty $, then $AT=TA$.

For a fixed $A\in B\left( X\right) $ and $\alpha \geq 0,$ we define the
class $\mathcal{D}_{A}^{\alpha }\left( 
\mathbb{R}
\right) $ of all operators $T\in B\left( X\right) $ for which the growth of
the function%
\begin{equation*}
t\rightarrow \left\Vert e^{tA}Te^{-tA}\right\Vert
\end{equation*}%
is at most polynomial in $t\in 
\mathbb{R}
$, explicitly, there exists a constant $C_{T}>0$ such that 
\begin{equation*}
\left\Vert e^{tA}Te^{-tA}\right\Vert \leq C_{T}\left( 1+\left\vert
t\right\vert \right) ^{\alpha },\text{ \ }\forall t\in 
\mathbb{R}
.
\end{equation*}%
Notice that $\mathcal{D}_{A}^{\alpha }\left( 
\mathbb{R}
\right) $ is a linear (not necessarily closed) subspace of $B\left( X\right) 
$ and contains the commutant of $A.$ In the case $\alpha =0,$ instead of $%
\mathcal{D}_{A}^{0}\left( 
\mathbb{R}
\right) $ we will use the notation $\mathcal{D}_{A}\left( 
\mathbb{R}
\right) .$ Notice also that $\mathcal{D}_{A}\left( 
\mathbb{R}
\right) $ is a normed (not necessarily closed) algebra with identity.

In Section 2, we give complete characterization of the class $\mathcal{D}%
_{A}^{\alpha }\left( 
\mathbb{R}
\right) $ in the case when the spectrum of $A$ consists of one point.
Section 3 contains results related to the decomposability of $A$. In the
case $0\leq \alpha <1,$ some estimates for the norm of the commutator $AT-TA$
are obtained in Section 4, where $T\in \mathcal{D}_{A}^{\alpha }\left( 
\mathbb{R}
\right) .$

\section{The class $\mathcal{D}_{A}^{\protect\alpha }\left( 
\mathbb{R}
\right) $}

In this section, we give complete characterization of the class $\mathcal{D}%
_{A}^{\alpha }\left( 
\mathbb{R}
\right) $ in the case when the spectrum of $A$ consists of one point. As
usual, $\sigma \left( T\right) $ will denote the spectrum of the operator $%
T\in B\left( X\right) .$ Throughout, $\left[ \alpha \right] $ denotes the
integer part of $\alpha \in 
\mathbb{R}
.$

Let $A\in B\left( X\right) $ and $\Delta _{A}$ be the inner derivation on $%
B\left( X\right) ;$%
\begin{equation*}
\Delta _{A}:T\mapsto AT-TA,\text{ \ }T\in B\left( X\right) .
\end{equation*}%
Then, we can write%
\begin{equation*}
\Delta _{A}^{n}\left( T\right) =\sum\limits_{k=0}^{n}\left( -1\right) ^{k}%
\binom{n}{k}A^{n-k}TA^{k}\text{ \ }\left( n\in 
\mathbb{N}
\right) .
\end{equation*}

We have the following:

\begin{theorem}
If the spectrum of the operator $A\in B\left( X\right) $ consists of one
point, then 
\begin{equation*}
\mathcal{D}_{A}^{\alpha }\left( 
\mathbb{R}
\right) =\ker \Delta _{A}^{\left[ \alpha \right] +1}.
\end{equation*}%
In particular, if $0\leq \alpha <1,$ then $\mathcal{D}_{A}^{\alpha }\left( 
\mathbb{R}
\right) =\left\{ A\right\} ^{\prime }.$
\end{theorem}

Before to prove this theorem, we first prove the following:

\begin{theorem}
If the spectrum of the operator $A\in B\left( X\right) $ lies on the real
line, then 
\begin{equation*}
\mathcal{D}_{A}^{\alpha }\left( 
\mathbb{R}
\right) =\ker \Delta _{A}^{\left[ \alpha \right] +1}.
\end{equation*}%
In particular, if $0\leq \alpha <1,$ then $\mathcal{D}_{A}^{\alpha }\left( 
\mathbb{R}
\right) =\left\{ A\right\} ^{\prime }.$
\end{theorem}

For related results see also, \cite{1,13}. For the proof of Theorem 2.2, we
need some preliminary results.

For an arbitrary $T\in B\left( X\right) $ and $x\in X$, we define $\rho
_{T}\left( x\right) $ to be the set of all $\lambda \in 
\mathbb{C}
$ for which there exists a neighborhood $U_{\lambda }$ of $\lambda $ with $%
u\left( z\right) $ analytic on $U_{\lambda }$ having values in $X$ such that 
$\left( zI-T\right) u\left( z\right) =x$ for all $z\in U_{\lambda }$. This
set is open and contains the resolvent set $\rho \left( T\right) $ of $T$.
By definition, the \textit{local spectrum} of $T$ at $x\in X$, denoted by $%
\sigma _{T}\left( x\right) $ is the complement of $\rho _{T}\left( x\right) $%
, so it is a compact subset of $\sigma \left( T\right) $. This object is
most tractable if the operator $T$ has the \textit{single-valued extension
property }(SVEP), i.e. for every open set $U$ in $%
\mathbb{C}
,$ the only analytic function $u:U\rightarrow X$ for which the equation $%
\left( zI-T\right) u\left( z\right) =0$ holds is the constant function $%
u\equiv 0$. If $T$ has SVEP, then $\sigma _{T}\left( x\right) \neq \emptyset
,$ whenever $x\in X\diagdown \left\{ 0\right\} $ \cite[Proposition 1.2.16]{9}%
. It can be seen that an operator $T\in B\left( X\right) $ having spectrum
without interior points has the SVEP. Ample information about local spectra
can be found in \cite{2,3,5,9}.

The \textit{local spectral radius }of $T\in B\left( X\right) $ at $x\in X$
is defined by 
\begin{equation*}
r_{T}\left( x\right) =\sup \left\{ \left\vert \lambda \right\vert :\lambda
\in \sigma _{T}\left( x\right) \right\} .
\end{equation*}%
If $T$ has SVEP, then%
\begin{equation*}
r_{T}\left( x\right) =\underset{n\rightarrow \infty }{\overline{\lim }}%
\left\Vert T^{n}x\right\Vert ^{\frac{1}{n}}
\end{equation*}%
\cite[Proposition 3.3.13]{9}.

Recall that a \textit{weight} \textit{function} (shortly a \textit{weight}) $%
\omega $ is a continuous function on $%
\mathbb{R}
$ such that $\omega \left( t\right) \geq 1$ and $\omega \left( t+s\right)
\leq \omega \left( t\right) \omega \left( s\right) $ for all $t,s\in 
\mathbb{R}
.$ For a weight function $\omega $, by $L_{\omega }^{1}\left( 
\mathbb{R}
\right) $ we will denote the Banach space of the functions $f\in L^{1}\left( 
\mathbb{R}
\right) $ with the norm%
\begin{equation*}
\left\Vert f\right\Vert _{1,\omega }=\int_{%
\mathbb{R}
}\left\vert f\left( t\right) \right\vert \omega \left( t\right) dt<\infty .
\end{equation*}%
The space $L_{\omega }^{1}\left( 
\mathbb{R}
\right) $ with convolution product and the norm $\left\Vert \cdot
\right\Vert _{1,\omega }$ is a commutative Banach algebra and is called 
\textit{Beurling algebra}. The dual space of $L_{\omega }^{1}\left( 
\mathbb{R}
\right) ,$ denoted by $L_{\omega }^{\infty }\left( 
\mathbb{R}
\right) $, is the space of all measurable functions $g$ on $%
\mathbb{R}
$ such that 
\begin{equation*}
\left\Vert g\right\Vert _{\omega ,\infty }:=\text{ess}\sup_{t\in 
\mathbb{R}
}\frac{\left\vert g\left( t\right) \right\vert }{\omega \left( t\right) }%
<\infty .
\end{equation*}%
The duality being implemented by the formula%
\begin{equation*}
\langle g,f\rangle =\int_{%
\mathbb{R}
}g\left( -t\right) f\left( t\right) dt\text{, \ }\forall f\in L_{\omega
}^{1}\left( 
\mathbb{R}
\right) ,\text{ }\forall g\in L_{\omega }^{\infty }\left( 
\mathbb{R}
\right) .
\end{equation*}%
We say that the weight $\omega $ is \textit{regular} if 
\begin{equation*}
\int_{%
\mathbb{R}
}\frac{\log \omega \left( t\right) }{1+t^{2}}dt<\infty .
\end{equation*}%
For example, $\omega \left( t\right) =\left( 1+\left\vert t\right\vert
\right) ^{\alpha }$ $\left( \alpha \geq 0\right) $\ is a regular weight and
is called \textit{polynomial weight}. If $\omega $ is a regular weight, then 
\begin{equation}
\lim_{t\rightarrow +\infty }\frac{\log \omega \left( t\right) }{t}%
=\lim_{t\rightarrow +\infty }\frac{\log \omega \left( -t\right) }{t}=0.
\label{2.1}
\end{equation}%
Consequently, the maximal ideal space of the algebra $L_{\omega }^{1}\left( 
\mathbb{R}
\right) $ can be identified with $%
\mathbb{R}
$ (for instance, see \cite{6,11,14}). The Gelfand transform of $f\in
L_{\omega }^{1}\left( 
\mathbb{R}
\right) $ is just the Fourier transform of $f.$ Moreover, the algebra $%
L_{\omega }^{1}\left( 
\mathbb{R}
\right) $ is regular in the Shilov sense \cite[Ch.VI]{6}. Notice also that $%
L_{\omega }^{1}\left( 
\mathbb{R}
\right) $ is Tauberian, that is, the set $\left\{ f\in L_{\omega }^{1}\left( 
\mathbb{R}
\right) :\text{supp}\widehat{f}\text{ is compact}\right\} $ is dense in $%
L_{\omega }^{1}\left( 
\mathbb{R}
\right) $ \cite[Ch.5]{11}. Below, we will assume that $\omega $ is a regular
weight.

Denote by $M_{\omega }\left( 
\mathbb{R}
\right) $ the Banach algebra (with respect to convolution product) of all
complex regular Borel measures on $%
\mathbb{R}
$ such that 
\begin{equation*}
\left\Vert \mu \right\Vert _{1,\omega }:=\int_{%
\mathbb{R}
}\omega \left( t\right) d\left\vert \mu \right\vert \left( t\right) <\infty .
\end{equation*}%
The algebra $L_{\omega }^{1}\left( 
\mathbb{R}
\right) $ is naturally identifiable with a closed ideal of $M_{\omega
}\left( 
\mathbb{R}
\right) .$ By $\widehat{f}$ and $\widehat{\mu }$, we will denote the Fourier
and the Fourier-Stieltjes transform of $f\in L_{\omega }^{1}\left( 
\mathbb{R}
\right) $ and $\mu \in M_{\omega }\left( 
\mathbb{R}
\right) $, respectively.

As usual, to any closed subset $S$ of $%
\mathbb{R}
,$ the following two closed ideals of $L_{\omega }^{1}\left( 
\mathbb{R}
\right) $ associated:%
\begin{equation*}
I_{\omega }\left( S\right) :=\left\{ f\in L_{\omega }^{1}\left( 
\mathbb{R}
\right) :\widehat{f}\left( S\right) =\left\{ 0\right\} \right\}
\end{equation*}%
and 
\begin{equation*}
J_{\omega }\left( S\right) :=\overline{\left\{ f\in L_{\omega }^{1}\left( 
\mathbb{R}
\right) :\text{supp}\widehat{f}\text{ is compact and supp}\widehat{f}\cap
S=\emptyset \right\} }.
\end{equation*}%
The ideals $J_{\omega }\left( S\right) $ and $I_{\omega }\left( S\right) $
are respectively, the smallest and the largest closed ideals in $L_{\omega
}^{1}\left( 
\mathbb{R}
\right) $ with hull $S$. When these two ideals coincide, the set $S$ is said
to be a \textit{set of} \textit{synthesis} for $L_{\omega }^{1}\left( 
\mathbb{R}
\right) $ (for instance, see \cite[Sect. 8.3]{8}).

Notice that $I_{\omega }\left( \left\{ \infty \right\} \right) =L_{\omega
}^{1}\left( 
\mathbb{R}
\right) $ and%
\begin{equation*}
J_{\omega }\left( \left\{ \infty \right\} \right) =\overline{\left\{ f\in
L_{\omega }^{1}\left( 
\mathbb{R}
\right) :\text{supp}\widehat{f}\text{ is compact}\right\} }.
\end{equation*}%
Since the algebra $L_{\omega }^{1}\left( 
\mathbb{R}
\right) $ is Tauberian, we have $I_{\omega }\left( \left\{ \infty \right\}
\right) =J_{\omega }\left( \left\{ \infty \right\} \right) $. Hence, $%
\left\{ \infty \right\} $ is a set of synthesis for $L_{\omega }^{1}\left( 
\mathbb{R}
\right) .$ Notice also that if $\omega \left( t\right) =\left( 1+\left\vert
t\right\vert \right) ^{\alpha }$ ($0\leq \alpha <1),$ then each point of $%
\mathbb{R}
$ is a set of synthesis for $L_{\omega }^{1}\left( 
\mathbb{R}
\right) $ \cite[Ch.6]{15}.

Let $M$ be a non-void subset of $L_{\omega }^{\infty }\left( 
\mathbb{R}
\right) $. A point $\lambda \in 
\mathbb{R}
$ is said to be \textit{Beurling spectrum }of $M$ if the character $%
e^{-i\lambda t}$ belongs to the weak$^{\ast }$-closed translation invariant
subspace of $L_{\omega }^{\infty }\left( 
\mathbb{R}
\right) $ generated by $M$. By sp$_{B}\left\{ M\right\} ,$ we will denote
the set of all Beurling spectrum of $M$. It is easy to verify that 
\begin{equation*}
\text{sp}_{B}\left\{ M\right\} =\text{hull}\left( \mathcal{I}_{\left\{
M\right\} }\right) \text{,}
\end{equation*}%
where%
\begin{equation*}
\mathcal{I}_{\left\{ M\right\} }=\left\{ f\in L_{\omega }^{1}\left( 
\mathbb{R}
\right) :f\ast g=0\text{, \ }\forall g\in M\right\}
\end{equation*}%
is a closed ideal of $L_{\omega }^{1}\left( 
\mathbb{R}
\right) $. Notice also that 
\begin{equation*}
\text{sp}_{B}\left\{ M\right\} =\overline{\underset{g\in M}{\bigcup }\text{sp%
}_{B}\left\{ g\right\} }\text{.}
\end{equation*}%
For $g\in L_{\omega }^{\infty }\left( 
\mathbb{R}
\right) ,$ we put $g^{\vee }\left( t\right) :=g\left( -t\right) .$ Clearly, 
\begin{equation*}
\text{sp}_{B}\left\{ g^{\vee }\right\} =\left\{ -\lambda :\lambda \in \text{%
sp}_{B}\left\{ g\right\} \right\} .
\end{equation*}

Recall that the \textit{Carleman transform} of $g\in L_{\omega }^{\infty
}\left( 
\mathbb{R}
\right) $ is defined as the analytic function $G\left( z\right) $ on $%
\mathbb{C}
\diagdown i%
\mathbb{R}
,$ given by 
\begin{equation*}
G\left( z\right) =\left\{ 
\begin{tabular}{l}
$\int\limits_{0}^{\infty }e^{-zt}{g\left( t\right) dt\text{, }\ }$%
\textnormal{Re}${z>0\text{;}}$ \\ 
${-}\int\limits_{-\infty }^{0}e^{-zt}{g\left( t\right) dt\text{, \ Re}z<0%
\text{.}}$%
\end{tabular}%
\right.
\end{equation*}%
It is known \cite{7} that $\lambda \in $sp$_{B}\left\{ g\right\} $ if and
only if the Carleman transform $G\left( z\right) $ of $g$ has no analytic
extension to a neighborhood of $i\lambda .$

Let $\omega $ be a weight function, $T\in B\left( X\right) $ and let%
\begin{equation*}
E_{T}^{\omega }:=\left\{ x\in X:\exists C>0,\text{ }\left\Vert
e^{tT}x\right\Vert \leq C\omega \left( t\right) \text{, \ }\forall t\in 
\mathbb{R}
\right\} .
\end{equation*}%
Then, $E_{T}^{\omega }$ is a linear (non-closed, in general) subspace of $X$%
. If $x\in E_{T}^{\omega }$, then for an arbitrary $\mu \in M_{\omega
}\left( 
\mathbb{R}
\right) ,$ we can define $x_{\mu }\in X$ by 
\begin{equation*}
x_{\mu }=\int_{%
\mathbb{R}
}e^{tT}xd\mu \left( t\right) \text{.}
\end{equation*}%
Clearly, $\mu \mapsto x_{\mu }$ is a bounded linear map from $M_{\omega
}\left( 
\mathbb{R}
\right) $ into $X$; 
\begin{equation*}
\left\Vert x_{\mu }\right\Vert \leq C\left\Vert \mu \right\Vert _{1,\omega }%
\text{, \ }\forall \mu \in M_{\omega }\left( 
\mathbb{R}
\right) .
\end{equation*}%
Further, from the identity%
\begin{equation*}
e^{tT}x_{\mu }=\int_{%
\mathbb{R}
}e^{\left( t+s\right) T}xd\mu \left( s\right) ,
\end{equation*}%
we can write 
\begin{eqnarray*}
\left\Vert e^{tT}x_{\mu }\right\Vert &\leq &\int_{%
\mathbb{R}
}\left\Vert e^{\left( t+s\right) T}x\right\Vert d\left\vert \mu \right\vert
\left( s\right) \\
&\leq &C\int_{%
\mathbb{R}
}\omega \left( t+s\right) d\left\vert \mu \right\vert \left( s\right) \\
&\leq &C\int_{%
\mathbb{R}
}\omega \left( t\right) \omega \left( s\right) d\left\vert \mu \right\vert
\left( s\right) \\
&=&C\left\Vert \mu \right\Vert _{1,\omega }\omega \left( t\right) \text{, \ }%
\forall t\in 
\mathbb{R}
.
\end{eqnarray*}%
This shows that $x_{\mu }\in E_{T}^{\omega }$ for every $\mu \in M_{\omega
}\left( 
\mathbb{R}
\right) .$ It is easy to check that 
\begin{equation*}
\left( x_{\mu }\right) _{\nu }=x_{\mu \ast \nu }\text{, \ }\forall \mu ,\nu
\in M_{\omega }\left( 
\mathbb{R}
\right) .
\end{equation*}%
It follows that if $x\in E_{T}^{\omega },$ then 
\begin{equation*}
I_{x}:=\left\{ f\in L_{\omega }^{1}\left( 
\mathbb{R}
\right) :x_{f}=0\right\}
\end{equation*}%
is a closed ideal of $L_{\omega }^{1}\left( 
\mathbb{R}
\right) ,$ where%
\begin{equation*}
x_{f}=\int_{%
\mathbb{R}
}f\left( t\right) e^{tT}xdt\text{.}
\end{equation*}

For a given $x\in E_{T}^{\omega },$ consider the function 
\begin{equation}
u\left( z\right) :=\left\{ 
\begin{tabular}{l}
$\int\limits_{0}^{\infty }e^{-zt}e^{tT}x{dt\text{, \ }\func{Re}z>0\text{;}}$
\\ 
${-}\int\limits_{-\infty }^{0}e^{-zt}e^{tT}x{dt\text{, \ }\func{Re}z<0\text{.%
}}$%
\end{tabular}%
\right.  \label{2.2}
\end{equation}%
It follows from (2.1) that $u\left( z\right) $ is a function analytic on $%
\mathbb{C}
\diagdown $i$%
\mathbb{R}
$. Let $a:=\func{Re}z>0.$ Then, for an arbitrary $s>0,$ we can write%
\begin{equation*}
\left( zI-T\right) \int\limits_{0}^{s}e^{-zt}e^{tT}x{dt}=-\int\limits_{0}^{s}%
\frac{d}{dt}e^{t\left( T-zI\right) }x{dt=x-}e^{s\left( T-zI\right) }x.
\end{equation*}%
Since 
\begin{equation*}
\left\Vert e^{s\left( T-zI\right) }x\right\Vert =e^{-as}\left\Vert
e^{sT}x\right\Vert \leq Ce^{-as}\omega \left( s\right)
\end{equation*}%
and 
\begin{equation*}
\lim_{s\rightarrow +\infty }e^{-as}\omega \left( s\right) =0,
\end{equation*}%
we have 
\begin{equation*}
\left( zI-T\right) u\left( z\right) =x,\text{ \ }\forall z\in 
\mathbb{C}
\text{ with }\func{Re}z>0.
\end{equation*}%
Similarly,%
\begin{equation*}
\left( zI-T\right) u\left( z\right) =x,\text{ \ }\forall z\in 
\mathbb{C}
\text{ with }\func{Re}z<0.
\end{equation*}%
Hence 
\begin{equation}
\left( zI-T\right) u\left( z\right) =x\text{, \ }\forall z\in 
\mathbb{C}
\diagdown i%
\mathbb{R}
.  \label{2.3}
\end{equation}%
This clearly implies that $\sigma _{T}\left( x\right) \subset i%
\mathbb{R}
.$

Thus we have the following:

\begin{proposition}
Let $\omega $ be a regular weight. Assume that $T\in B\left( X\right) $ and $%
x\in X$ satisfy the condition $\left\Vert e^{tT}x\right\Vert \leq C\omega
\left( t\right) $ for all $t\in 
\mathbb{R}
$ and for some $C>0.$ Then, $\sigma _{T}\left( x\right) \subset i%
\mathbb{R}
.$
\end{proposition}

Now, assume that $T$ has SVEP. We claim that $\sigma _{T}\left( x\right) $
consists of all $\lambda \in i%
\mathbb{R}
$ for which the function $u\left( z\right) $ has no analytic extension to a
neighborhood of $\lambda $. To see this, let $v\left( z\right) $ be the
analytic extension of $u\left( z\right) $ to a neighborhood $U_{\lambda }$
of $\lambda \in i%
\mathbb{R}
.$ It follows from the identity (2.3) that the function%
\begin{equation*}
w\left( z\right) :=\left( zI-T\right) v\left( z\right) -x
\end{equation*}%
vanishes on $U_{\lambda }^{+}:=\left\{ z\in U_{\lambda }:\func{Re}%
z>0\right\} $ and on $U_{\lambda }^{-}:=\left\{ z\in U_{\lambda }:\func{Re}%
z<0\right\} .$ By uniqueness theorem, $w\left( z\right) =0$ for all $z\in
U_{\lambda }.$ So we have 
\begin{equation*}
\left( zI-T\right) v\left( z\right) =x\text{, \ }\forall z\in U_{\lambda }.
\end{equation*}%
This shows that $\lambda \in \rho _{T}\left( x\right) .$ If $\lambda \in
\rho _{T}\left( x\right) \cap i%
\mathbb{R}
,$ then there exists a neighborhood $U_{\lambda }$ of $\lambda $ with $%
v\left( z\right) $ analytic on $U_{\lambda }$ having values in $X$ such that 
\begin{equation*}
\left( zI-T\right) v\left( z\right) =x\text{, \ }\forall z\in U_{\lambda }.
\end{equation*}%
By (2.3), 
\begin{equation*}
\left( zI-T\right) \left( u\left( z\right) -v\left( z\right) \right) =0,%
\text{ \ }\forall z\in U_{\lambda }^{+},\text{ }\forall z\in U_{\lambda
}^{-}.
\end{equation*}%
Since $T$ has SVEP, we have 
\begin{equation*}
u\left( z\right) =v\left( z\right) \text{, \ }\forall z\in U_{\lambda }^{+},%
\text{ }\forall z\in U_{\lambda }^{-}.
\end{equation*}%
This shows that $u\left( z\right) $ can be analytically extended to a
neighborhood of $\lambda $.

Let $x\in E_{T}^{\omega }.$ For a given $\varphi \in X^{\ast },$ define a
function $\varphi _{x}$ on $%
\mathbb{R}
$ by 
\begin{equation*}
\varphi _{x}\left( t\right) =\langle \varphi ,e^{tT}x\rangle .
\end{equation*}%
Then, $\varphi _{x}$ is continuous and 
\begin{equation*}
\left\vert \varphi _{x}\left( t\right) \right\vert \leq C\left\Vert \varphi
\right\Vert \omega \left( t\right) \text{, \ }\forall t\in 
\mathbb{R}
.
\end{equation*}%
Consequently, $\varphi _{x}\in L_{\omega }^{\infty }\left( 
\mathbb{R}
\right) $. Taking into account the identity (2.2), we have 
\begin{equation*}
\langle \varphi ,u\left( z\right) \rangle =\left\{ 
\begin{tabular}{l}
$\int\limits_{0}^{\infty }e^{-zt}{\varphi _{x}\left( t\right) dt\text{, \ }%
\func{Re}z>0\text{;}}$ \\ 
${-}\int\limits_{-\infty }^{0}e^{-zt}{\varphi _{x}\left( t\right) dt\text{,
\ }\func{Re}z<0}$.%
\end{tabular}%
\right.
\end{equation*}%
This shows that the function $z\rightarrow \langle \varphi ,u\left( z\right)
\rangle $ is the Carleman transform of ${\varphi _{x}.}$ It follows that 
\begin{equation*}
i\text{sp}_{B}\left\{ \varphi _{x}\right\} \subseteq \sigma _{T}\left(
x\right) \text{, \ }\forall \varphi \in X^{\ast },
\end{equation*}%
and so 
\begin{equation*}
\overline{\bigcup_{\varphi \in X^{\ast }}\text{sp}_{B}\left\{ \varphi
_{x}\right\} }\subseteq -i\sigma _{T}\left( x\right) .
\end{equation*}%
To show the reverse inclusion, assume that $\lambda _{0}\in 
\mathbb{R}
$ and 
\begin{equation*}
\lambda _{0}\notin \overline{\bigcup_{\varphi \in X^{\ast }}\text{sp}%
_{B}\left\{ \varphi _{x}\right\} }\text{.}
\end{equation*}%
Then, there exist a neighborhood $U$ of $\overline{\bigcup_{\varphi \in
X^{\ast }}\text{sp}_{B}\left\{ \varphi _{x}\right\} }$ and $\varepsilon >0$
such that $\left( \lambda _{0}-\varepsilon ,\lambda _{0}+\varepsilon \right)
\cap \overline{U}=\emptyset .$ Since the algebra $L_{\omega }^{1}\left( 
\mathbb{R}
\right) $ is regular, there exists a function $f\in L_{\omega }^{1}\left( 
\mathbb{R}
\right) $ such that $\widehat{f}=1$ on $\left[ \lambda _{0}-\varepsilon
/2,\lambda _{0}+\varepsilon /2\right] $ and $\widehat{f}=0$ on $\overline{U}%
. $ Notice that $\widehat{f}$ vanishes in a neighborhood of sp$_{B}\left\{
\varphi _{x}\right\} $ and supp$\widehat{f}\subseteq \left[ \lambda
_{0}-\varepsilon ,\lambda _{0}+\varepsilon \right] .$ Consequently, $f$
belongs to the smallest ideal of $L_{\omega }^{1}\left( 
\mathbb{R}
\right) $ whose hull is sp$_{B}\left\{ \varphi _{x}\right\} .$ Since 
\begin{equation*}
\text{sp}_{B}\left\{ \varphi _{x}\right\} =\text{hull}\left\{ f\in L_{\omega
}^{1}\left( 
\mathbb{R}
\right) :f\ast \varphi _{x}=0\right\} ,
\end{equation*}%
we have $\left( \lambda _{0}-\varepsilon /2,\lambda _{0}+\varepsilon
/2\right) \subset 
\mathbb{R}
\setminus $sp$_{B}\left\{ \varphi _{x}\right\} $ for every $\varphi \in
X^{\ast }.$ It follows that the function $z\rightarrow \langle \varphi
,u\left( z\right) \rangle $ can be analytically extended to $i\left( \lambda
_{0}-\varepsilon /2,\lambda _{0}+\varepsilon /2\right) $ for every $\varphi
\in X^{\ast }.$ Hence, $u\left( z\right) $ can be analytically extended to $%
i\left( \lambda _{0}-\varepsilon /2,\lambda _{0}+\varepsilon /2\right) $ and
therefore, $i\lambda _{0}\notin \sigma _{T}\left( x\right) $ or $\lambda
_{0}\notin -i\sigma _{T}\left( x\right) .$ Thus we have 
\begin{equation*}
\overline{\bigcup_{\varphi \in X^{\ast }}\text{sp}_{B}\left\{ \varphi
_{x}\right\} }=-i\sigma _{T}\left( x\right) .
\end{equation*}%
Further, it is easy to check that 
\begin{equation*}
I_{x}=\bigcap\limits_{\varphi \in X^{\ast }}\mathcal{I}_{\left\{ \varphi
_{x}^{\vee }\right\} },
\end{equation*}%
where%
\begin{equation*}
\mathcal{I}_{\left\{ \varphi _{x}^{\vee }\right\} }:=\left\{ f\in L_{\omega
}^{1}\left( 
\mathbb{R}
\right) :f\ast \varphi _{x}^{\vee }=0\right\} .
\end{equation*}%
Taking into account that 
\begin{equation*}
\text{sp}_{B}\left\{ \varphi _{x}^{\vee }\right\} =\left\{ -\lambda :\lambda
\in \text{sp}_{B}\left\{ \varphi _{x}\right\} \right\} ,
\end{equation*}%
we can write 
\begin{equation*}
\text{hull}\left( I_{x}\right) =\overline{\bigcup_{\varphi \in X^{\ast }}%
\text{hull}\left( \mathcal{I}_{\left\{ \varphi _{x}^{\vee }\right\} }\right) 
}=\overline{\bigcup_{\varphi \in X^{\ast }}\text{sp}_{B}\left\{ \varphi
_{x}^{\vee }\right\} }=i\sigma _{T}\left( x\right) .
\end{equation*}

Thus we have the following:

\begin{proposition}
Let $\omega $ be a regular weight. Assume that $T\in B\left( X\right) $ has
SVEP and $x\in X$ satisfies the condition $\left\Vert e^{tT}x\right\Vert
\leq C\omega \left( t\right) $ for all $t\in 
\mathbb{R}
$ and for some $C>0.$ Then,%
\begin{equation*}
i\sigma _{T}\left( x\right) =\mathnormal{hull}\left( I_{x}\right) .
\end{equation*}
\end{proposition}

For $f\in L_{\omega }^{1}\left( 
\mathbb{R}
\right) $ and $s\in 
\mathbb{R}
,$ let $f_{s}\left( t\right) :=f\left( t-s\right) .$ Let $e_{n}:=2n\chi _{%
\left[ -1/n,1/n\right] }$ $\left( n\in 
\mathbb{N}
\right) $, where $\chi _{\left[ -1/n,1/n\right] }$ is the characteristic
function of the interval $\left[ -1/n,1/n\right] .$ If $K:=\sup_{t\in \left[
-1,1\right] }\omega \left( t\right) ,$ then $\left\Vert e_{n}\right\Vert
_{\omega }\leq K$ for all $n\in 
\mathbb{N}
.$ On the other hand, by continuity of the mapping $s\mapsto f_{s}$, we have 
\begin{equation*}
\lim_{n\rightarrow \infty }\left\Vert f\ast e_{n}-f\right\Vert _{1,\omega
}=0.
\end{equation*}%
Consequently, $\left\{ e_{n}\right\} $ is a bounded approximate identity
(b.a.i.) for $L_{\omega }^{1}\left( 
\mathbb{R}
\right) .$ If $x\in E_{T}^{\omega },$ then from the identity 
\begin{equation*}
x_{e_{n}}-x=\int_{-1/n}^{1/n}e_{n}\left( t\right) \left( e^{tT}x-x\right) dt,
\end{equation*}%
it follows that $x_{e_{n}}\rightarrow x.$ Similarly, $x_{f\ast
e_{n}}\rightarrow x_{f}$ for all $f\in L_{\omega }^{1}\left( 
\mathbb{R}
\right) .$

\begin{proposition}
Let $\omega $ be a regular weight and $x\in X.$ Assume that $T\in B\left(
X\right) $ has SVEP and $\left\Vert e^{tT}x\right\Vert \leq C\omega \left(
t\right) $ for all $t\in 
\mathbb{R}
$ and for some $C>0.$ For an arbitrary $f\in L_{\omega }^{1}\left( 
\mathbb{R}
\right) $, the following assertions hold:

$a)$ If $x_{f}=0,$ then $\widehat{f}$ vanishes on $i\sigma _{T}\left(
x\right) .$

$b)$ If $\widehat{f}$ vanishes in a neighborhood of $i\sigma _{T}\left(
x\right) ,$ then $x_{f}=0$.

$c)$ If $\widehat{f}=1$ in a neighborhood of $i\sigma _{T}\left( x\right) ,$
then $x_{f}=x$.
\end{proposition}

\begin{proof}
a) By Proposition 2.4, $i\sigma _{T}\left( x\right) =$hull$\left(
I_{x}\right) $ and therefore $I_{x}\subseteq I_{\omega }\left( i\sigma
_{T}\left( x\right) \right) .$ This clearly implies a).

b) Let $g\in L_{\omega }^{1}\left( 
\mathbb{R}
\right) $ be such that supp$\widehat{g}$ is compact. Then, $f\ast g\in
J_{\omega }\left( i\sigma _{T}\left( x\right) \right) $ and therefore, $%
f\ast g\in I_{x}.$ So we have $x_{f\ast g}=0.$ Since the algebra $L_{\omega
}^{1}\left( 
\mathbb{R}
\right) $ is Tauberian, $x_{f\ast g}=0$ for all $g\in L_{\omega }^{1}\left( 
\mathbb{R}
\right) .$ It follows that $x_{f\ast e_{n}}=0$ for all $n,$ where $\left\{
e_{n}\right\} $ be a b.a.i. for $L_{\omega }^{1}\left( 
\mathbb{R}
\right) .$ As $n\rightarrow \infty ,$ we have $x_{f}=0.$

c) Since the Fourier transform of $f\ast e_{n}-e_{n}$ vanishes in a
neighborhood of $i\sigma _{T}\left( x\right) ,$ by b), $x_{f\ast
e_{n}}=x_{e_{n}}$. As $n\rightarrow \infty ,$ we have $x_{f}=x.$
\end{proof}

By $S\left( 
\mathbb{R}
\right) ,$ we denote the set of all rapidly decreasing functions on $%
\mathbb{R}
,$ i.e. the set of all infinitely differentiable functions $\phi $ on $%
\mathbb{R}
$ such that 
\begin{equation*}
\lim_{\left\vert t\right\vert \rightarrow \infty }\left\vert t^{n}\phi
^{\left( k\right) }\left( t\right) \right\vert =0\text{, \ }\forall
n,k=0,1,2,....
\end{equation*}%
(in this definition, $n$ can be replaced by any $\alpha \geq 0$)$.$ It can
be seen that if $\omega $ is a polynomial weight, then $S\left( 
\mathbb{R}
\right) \subset L_{\omega }^{1}\left( 
\mathbb{R}
\right) .$

\begin{lemma}
Assume that $T\in B\left( X\right) $ and $x\in X$ satisfy the condition $%
\left\Vert e^{tT}x\right\Vert \leq C\left( 1+\left\vert t\right\vert \right)
^{\alpha }$ $\left( \alpha \geq 0\right) $ for all $t\in 
\mathbb{R}
$ and for some $C>0.$ Then, for an arbitrary $\phi \in S\left( 
\mathbb{R}
\right) ,$ we have%
\begin{equation*}
x_{\phi ^{\left( k\right) }}=\left( -1\right) ^{k}T^{k}x_{\phi }\text{, \ }%
\forall k\in 
\mathbb{N}
.
\end{equation*}
\end{lemma}

\begin{proof}
For an arbitrary $a,b\in 
\mathbb{R}
$ $\left( a<b\right) ,$ we can write 
\begin{equation*}
\int_{a}^{b}\phi ^{\prime }\left( t\right) e^{tT}xdt=\phi \left( b\right)
e^{bT}x-\phi \left( a\right) e^{aT}x-T\int_{a}^{b}\phi \left( t\right)
e^{tT}xdt.
\end{equation*}%
On the other hand, 
\begin{eqnarray*}
\left\Vert \phi \left( b\right) e^{bT}x-\phi \left( a\right)
e^{aT}x\right\Vert &\leq &C\left\vert \phi \left( b\right) \right\vert
\left( 1+\left\vert b\right\vert \right) ^{\alpha }+C\left\vert \phi \left(
a\right) \right\vert \left( 1+\left\vert a\right\vert \right) ^{\alpha } \\
&\leq &2^{\alpha }C\left( \left\vert \phi \left( b\right) \right\vert
+\left\vert \phi \left( b\right) \right\vert \left\vert b\right\vert
^{\alpha }+\left\vert \phi \left( a\right) \right\vert +\left\vert \phi
\left( a\right) \right\vert \left\vert a\right\vert ^{\alpha }\right) .
\end{eqnarray*}%
Since $\phi \in S\left( 
\mathbb{R}
\right) ,$ it follows that 
\begin{equation*}
\underset{a\rightarrow -\infty }{\lim_{b\rightarrow +\infty }}\left\Vert
\phi \left( b\right) e^{bT}x-\phi \left( a\right) e^{aT}x\right\Vert =0.
\end{equation*}%
Hence $x_{\phi ^{\prime }}=-Tx_{\phi }.$ By induction we obtain our result.
\end{proof}

Next, we have the following:

\begin{proposition}
Let $\omega \left( t\right) =\left( 1+\left\vert t\right\vert \right)
^{\alpha }$ $\left( \alpha \geq 0\right) .$ Assume that $T\in B\left(
X\right) $ has SVEP\ and $x\in X$ satisfies the condition $\left\Vert
e^{tT}x\right\Vert \leq C\omega \left( t\right) $ for all $t\in 
\mathbb{R}
$ and for some $C>0.$ If $\sigma _{T}\left( x\right) =\left\{ 0\right\} $,
then for an arbitrary $f\in L_{\omega }^{1}\left( 
\mathbb{R}
\right) ,$ we have%
\begin{equation*}
x_{f}=\widehat{f}\left( 0\right) x+\frac{\widehat{f}^{\prime }\left(
0\right) }{1!}\left( iT\right) x+...+\frac{\widehat{f}^{\left( k\right)
}\left( 0\right) }{k!}\left( iT\right) ^{k}x,
\end{equation*}%
where $k=[\alpha ].$ In particular, we have $T^{k+1}x=0.$
\end{proposition}

\begin{proof}
We know (for instance, see \cite[Ch.VI, \S 41]{6} and \cite[Theorem 3.2]{19}%
) that if $f\in L_{\omega }^{1}\left( 
\mathbb{R}
\right) ,$ then the first $k$ derivatives of the Fourier transform of $f$
exist and 
\begin{equation*}
J_{\omega }\left( \left\{ 0\right\} \right) =\left\{ f\in L_{\omega
}^{1}\left( 
\mathbb{R}
\right) :\widehat{f}\left( 0\right) =\widehat{f}^{\prime }\left( 0\right)
=...=\widehat{f}^{\left( k\right) }\left( 0\right) =0\right\} ,
\end{equation*}%
where $k=\left[ \alpha \right] $. Recall that $J_{\omega }\left( \left\{
0\right\} \right) $ is the smallest closed ideal of $L_{\omega }^{1}\left( 
\mathbb{R}
\right) $ whose hull is $\left\{ 0\right\} .$ On the other hand, by
Proposition 2.4, hull$\left( I_{x}\right) =\left\{ 0\right\} .$ Hence we
have $J_{\omega }\left( \left\{ 0\right\} \right) \subseteq I_{x}$.

Let $\phi \in S\left( 
\mathbb{R}
\right) $ be such that $\widehat{\phi }\left( \lambda \right) =1$ in a
neighborhood of $0.$ For a given $f\in L_{\omega }^{1}\left( 
\mathbb{R}
\right) ,$ consider the function 
\begin{equation*}
g:=f-\widehat{f}\left( 0\right) \phi -\frac{\widehat{f}^{\prime }\left(
0\right) }{1!}\left( -i\right) \phi ^{\prime }-...-\frac{\widehat{f}^{\left(
k\right) }\left( 0\right) }{k!}\left( -i\right) ^{k}\phi ^{\left( k\right) }.
\end{equation*}%
As%
\begin{equation*}
\widehat{\phi ^{\left( k\right) }}\left( \lambda \right) =\left( i\lambda
\right) ^{k}\widehat{\phi }\left( \lambda \right) ,
\end{equation*}%
we have%
\begin{equation*}
\widehat{g}\left( \lambda \right) =\widehat{f}\left( \lambda \right) -\left[ 
\widehat{f}\left( 0\right) +\frac{\widehat{f}^{\prime }\left( 0\right) }{1!}%
\lambda +...+\frac{\widehat{f}^{\left( k\right) }\left( 0\right) }{k!}%
\lambda ^{k}\right] \widehat{\phi }\left( \lambda \right) .
\end{equation*}%
It can be seen that the first $k$ derivatives of $\widehat{g}$ at $0$ are
zero and therefore $g\in J_{\omega }\left( \left\{ 0\right\} \right) .$
Consequently, $g\in I_{x}$ and so $x_{g}=0.$ On the other hand, by Lemma 2.6
and Proposition 2.5, 
\begin{equation*}
x_{\phi ^{\left( k\right) }}=\left( -1\right) ^{k}T^{k}x
\end{equation*}%
which implies%
\begin{equation*}
x_{g}=x_{f}-\widehat{f}\left( 0\right) x-\frac{\widehat{f}^{\prime }\left(
0\right) }{1!}\left( iT\right) x-...-\frac{\widehat{f}^{\left( k\right)
}\left( 0\right) }{k!}\left( iT\right) ^{k}x.
\end{equation*}%
Hence, 
\begin{equation*}
x_{f}=\widehat{f}\left( 0\right) x+\frac{\widehat{f}^{\prime }\left(
0\right) }{1!}\left( iT\right) x+...+\frac{\widehat{f}^{\left( k\right)
}\left( 0\right) }{k!}\left( iT\right) ^{k}x.
\end{equation*}%
If $f:=\phi ^{\left( k+1\right) },$ then as 
\begin{equation*}
\widehat{f}\left( 0\right) =\widehat{f}^{\prime }\left( 0\right) =...=%
\widehat{f}^{\left( k\right) }\left( 0\right) =0,
\end{equation*}%
we get 
\begin{equation*}
0=x_{f}=\left( -1\right) ^{k+1}T^{k+1}x.
\end{equation*}
\end{proof}

Now, we are in a position to prove Theorem 2.2.

\begin{proof}[Proof of Theorem 2.2]
If $T\in \mathcal{D}_{A}^{\alpha }\left( 
\mathbb{R}
\right) ,$ then as 
\begin{equation*}
e^{tA}Te^{-tA}=e^{t\Delta _{A}}\left( T\right) ,
\end{equation*}
we have%
\begin{equation*}
\left\Vert e^{t\Delta _{A}}\left( T\right) \right\Vert \leq C_{T}\left(
1+\left\vert t\right\vert \right) ^{\alpha }\text{, \ }\forall t\in 
\mathbb{R}
.
\end{equation*}%
By Proposition 2.3, $\sigma _{\Delta _{A}}\left( T\right) \subset i%
\mathbb{R}
.$ Further, since 
\begin{equation*}
\sigma \left( \Delta _{A}\right) =\left\{ \lambda -\mu :\lambda ,\mu \in
\sigma \left( A\right) \right\}
\end{equation*}%
\cite[Theorem 3.5.1]{9} and $\sigma \left( A\right) \subset 
\mathbb{R}
,$ we have $\sigma \left( \Delta _{A}\right) \subset 
\mathbb{R}
$. Therefore, $\Delta _{A}$ has SVEP. On the other hand, as 
\begin{equation*}
\sigma _{\Delta _{A}}\left( T\right) \subseteq \sigma \left( \Delta
_{A}\right) \subset 
\mathbb{R}
,
\end{equation*}%
we obtain that 
\begin{equation*}
\sigma _{\Delta _{A}}\left( T\right) \subseteq 
\mathbb{R}
\cap i%
\mathbb{R}
=\left\{ 0\right\} .
\end{equation*}%
Since $\Delta _{A}$ has SVEP, $\sigma _{\Delta _{A}}\left( T\right) \neq
\emptyset $, so that $\sigma _{\Delta _{A}}\left( T\right) =\left\{
0\right\} .$ Applying now Proposition 2.7 to the operator $\Delta _{A}$ on
the space $B\left( X\right) ,$ we get 
\begin{equation*}
\Delta _{A}^{[\alpha ]+1}\left( T\right) =0.
\end{equation*}%
For the reverse inclusion, assume that $T\in B\left( X\right) $ satisfies
the equation $\Delta _{A}^{n}\left( T\right) =0$ for some $n\in 
\mathbb{N}
.$ Then, we can write%
\begin{eqnarray*}
\left\Vert e^{tA}Te^{-tA}\right\Vert &=&\left\Vert e^{t\Delta _{A}}\left(
T\right) \right\Vert \\
&=&\left\Vert I+\frac{t\Delta _{A}\left( T\right) }{1!}+...+\frac{%
t^{n-1}\Delta _{A}^{n-1}\left( T\right) }{\left( n-1\right) !}\right\Vert \\
&=&O\left( 1+\left\vert t\right\vert \right) ^{n-1}.
\end{eqnarray*}%
This shows that $T\in \mathcal{D}_{A}^{n-1}\left( 
\mathbb{R}
\right) .$ The proof is complete.
\end{proof}

Next, we will prove Theorem 2.1.

\begin{proof}[Proof of Theorem 2.1]
Assume that $\sigma \left( A\right) =\left\{ \lambda \right\} .$ If $T\in 
\mathcal{D}_{A}^{\alpha }\left( 
\mathbb{R}
\right) ,$ then $T\in \mathcal{D}_{B}^{\alpha }\left( 
\mathbb{R}
\right) ,$ where $B=A-\lambda I.$ Since $\sigma \left( B\right) =\left\{
0\right\} $ and $\Delta _{B}^{n}\left( T\right) =\Delta _{A}^{n}\left(
T\right) $ $\left( \forall n\in 
\mathbb{N}
\right) $, by Theorem 2.2 we obtain our result.
\end{proof}

Note that if $\alpha \geq 1,$ then $\mathcal{D}_{A}^{\alpha }\left( 
\mathbb{R}
\right) \neq \left\{ A\right\} ^{\prime },$ in general. To see this, let $%
A=\left( 
\begin{array}{cc}
0 & 0 \\ 
1 & 0%
\end{array}%
\right) $ and $T=\left( 
\begin{array}{cc}
1 & 0 \\ 
0 & 0%
\end{array}%
\right) $ be two $2\times 2$ matrices on $2-$dimensional Hilbert space. As $%
A^{2}=0,$ we have $\sigma \left( A\right) =\left\{ 0\right\} $ and 
\begin{equation*}
e^{tA}Te^{-tA}=\left( I+tA\right) T\left( I-tA\right) =\left( 
\begin{array}{cc}
1 & 0 \\ 
t & 0%
\end{array}%
\right) \text{, \ }\forall t\in 
\mathbb{R}
.
\end{equation*}%
Since%
\begin{equation*}
\left\Vert e^{tA}Te^{-tA}\right\Vert =\left( 1+\left\vert t\right\vert
^{2}\right) ^{\frac{1}{2}},
\end{equation*}%
we have $T\in \mathcal{D}_{A}^{1}\left( 
\mathbb{R}
\right) ,$ but $AT\neq TA.$

For a given $\alpha \geq 0,$ we define the class $\mathcal{D}_{A}^{\alpha
}\left( 
\mathbb{Z}
\right) $ of all operators $T\in B\left( X\right) $ for which there exists a
constant $C_{T}>0$ such that 
\begin{equation*}
\left\Vert e^{nA}Te^{-nA}\right\Vert \leq C_{T}\left( 1+\left\vert
n\right\vert \right) ^{\alpha }\text{, \ }\forall n\in 
\mathbb{Z}
.
\end{equation*}%
Clearly, $\mathcal{D}_{A}^{\alpha }\left( 
\mathbb{R}
\right) \subseteq \mathcal{D}_{A}^{\alpha }\left( 
\mathbb{Z}
\right) .$ We claim that $\mathcal{D}_{A}^{\alpha }\left( 
\mathbb{R}
\right) =\mathcal{D}_{A}^{\alpha }\left( 
\mathbb{Z}
\right) .$ Indeed, if $T\in \mathcal{D}_{A}^{\alpha }\left( 
\mathbb{Z}
\right) $ and $t\in 
\mathbb{R}
,$ then $t=n+r,$ where $n\in 
\mathbb{Z}
$, $\left\vert r\right\vert <1$ and $\left\vert n\right\vert \leq \left\vert
t\right\vert .$ Consequently, we can write 
\begin{eqnarray*}
\left\Vert e^{tT}x\right\Vert &=&\left\Vert e^{rT}e^{nT}x\right\Vert \\
&\leq &e^{\left\Vert T\right\Vert }\left\Vert e^{nT}x\right\Vert \\
&\leq &C_{T}e^{\left\Vert T\right\Vert }\left( 1+\left\vert n\right\vert
\right) ^{\alpha } \\
&\leq &C_{T}e^{\left\Vert T\right\Vert }\left( 1+\left\vert t\right\vert
\right) ^{\alpha }.
\end{eqnarray*}%
Therefore, in Theorems 2.1 and 2.2, the class $\mathcal{D}_{A}^{\alpha
}\left( 
\mathbb{R}
\right) $ can be replaced by $\mathcal{D}_{A}^{\alpha }\left( 
\mathbb{Z}
\right) .$

As a consequence of Proposition 2.7, we will need the following:

\begin{corollary}
Assume that $T\in B\left( X\right) $ has SVEP\ and $x\in X$ satisfies the
condition 
\begin{equation*}
\left\Vert e^{nT}x\right\Vert \leq C\left( 1+\left\vert n\right\vert \right)
^{\alpha }\text{ }\left( \alpha \geq 0\right) ,
\end{equation*}%
for all $n\in 
\mathbb{Z}
$ and for some $C>0.$ If $\sigma _{T}\left( x\right) =\left\{ 0\right\} $,
then $T^{\left[ \alpha \right] +1}x=0.$
\end{corollary}

By $K\left( X\right) $ we will denote the space of compact operators on a
Banach space $X.$

Next, we have the following:

\begin{proposition}
Assume that the spectrum of the operator $A\in B\left( X\right) $ consists
of one point. If $K\left( X\right) \subset \mathcal{D}_{A}^{\alpha }\left( 
\mathbb{R}
\right) ,$ then%
\begin{equation*}
\mathcal{D}_{A}^{2\left[ \alpha \right] }\left( 
\mathbb{R}
\right) =B\left( X\right) .
\end{equation*}
\end{proposition}

\begin{proof}
We have 
\begin{equation*}
\left\Vert e^{n\Delta _{A}}\left( T\right) \right\Vert =\left\Vert
e^{nA}Te^{-nA}\right\Vert \leq C_{T}\left( 1+n\right) ^{\alpha }\text{, \ }%
\forall T\in K\left( X\right) ,\text{ }\forall n\in 
\mathbb{N}
.
\end{equation*}%
Applying uniform boundedness principle to the sequence of operators 
\begin{equation*}
\left\{ \frac{1}{\left( 1+n\right) ^{\alpha }}e^{n\Delta _{A}}\right\}
_{n\in 
\mathbb{N}
},
\end{equation*}%
we obtain existence of a constant $C>0$ such that 
\begin{equation*}
\left\Vert e^{nA}Te^{-nA}\right\Vert \leq C\left( 1+n\right) ^{\alpha
}\left\Vert T\right\Vert \text{, \ }\forall T\in K\left( X\right) ,\text{ }%
\forall n\in 
\mathbb{N}
.
\end{equation*}%
Consequently, we have 
\begin{equation*}
\left\Vert e^{nA}Te^{-nA}\right\Vert \leq C\left( 1+\left\vert n\right\vert
\right) ^{\alpha }\left\Vert T\right\Vert \text{, \ }\forall T\in K\left(
X\right) ,\text{ }\forall n\in 
\mathbb{Z}
.
\end{equation*}%
For a given $x\in X$ and $\varphi \in X^{\ast }$, let $x\otimes \varphi $ be
the one dimensional operator on $X;$%
\begin{equation*}
x\otimes \varphi :y\mapsto \varphi \left( y\right) x\text{, \ }y\in X.
\end{equation*}%
By taking $T=x\otimes \varphi $ in the preceding inequality, we can write 
\begin{equation*}
\left\Vert e^{nA}x\right\Vert \left\Vert e^{-nA^{\ast }}\varphi \right\Vert
\leq C\left( 1+\left\vert n\right\vert \right) ^{\alpha }\left\Vert
x\right\Vert \left\Vert \varphi \right\Vert \text{, \ }\forall x\in X,\text{ 
}\forall \varphi \in X^{\ast },
\end{equation*}%
which implies 
\begin{equation*}
\left\Vert e^{nA}\right\Vert \left\Vert e^{-nA}\right\Vert \leq C\left(
1+\left\vert n\right\vert \right) ^{\alpha }\text{, \ }\forall n\in 
\mathbb{Z}
.
\end{equation*}%
Now, assume that $\sigma \left( A\right) =\left\{ \lambda \right\} .$ Then
as $\sigma \left( e^{n\left( A-\lambda I\right) }\right) =\left\{ 1\right\}
, $ we have 
\begin{equation*}
\left\Vert e^{n\left( A-\lambda I\right) }\right\Vert \geq 1,\text{ \ }%
\forall n\in 
\mathbb{Z}
,
\end{equation*}%
so that 
\begin{eqnarray*}
\left\Vert e^{n\left( A-\lambda I\right) }\right\Vert &\leq &\left\Vert
e^{n\left( A-\lambda I\right) }\right\Vert \left\Vert e^{-n\left( A-\lambda
I\right) }\right\Vert \\
&=&\left\Vert e^{nA}\right\Vert \left\Vert e^{-nA}\right\Vert \\
&\leq &C\left( 1+\left\vert n\right\vert \right) ^{\alpha }\text{, \ }%
\forall n\in 
\mathbb{Z}
.
\end{eqnarray*}%
Thus, we obtain that%
\begin{equation*}
\left\Vert e^{n\left( A-\lambda I\right) }\right\Vert \leq C\left(
1+\left\vert n\right\vert \right) ^{\alpha }\text{, \ }\forall n\in 
\mathbb{Z}
.
\end{equation*}%
By Corollary 2.8, $\left( A-\lambda I\right) ^{k+1}=0,$ where $k=\left[
\alpha \right] .$ If $N:=A-\lambda I,$ then $A=\lambda I+N$, where $N$ is a
nilpotent of degree $\leq k+1.$ Further, for an arbitrary $T\in B\left(
X\right) $ and $t\in 
\mathbb{R}
,$ from the identity 
\begin{eqnarray*}
e^{tA}Te^{-tA} &=&e^{t\left( \lambda I+N\right) }Te^{-t\left( \lambda
I+N\right) } \\
&=&\left( I+\frac{tN}{1!}+...+\frac{t^{k}N^{k}}{k!}\right) T\left( I-\frac{tN%
}{1!}+...+\left( -1\right) ^{k}\frac{t^{k}N^{k}}{k!}\right) ,
\end{eqnarray*}%
we can write 
\begin{equation*}
e^{tA}Te^{-tA}=T+tf_{1}\left( N,T\right) +...+t^{2k}f_{2k}\left( N,T\right) ,
\end{equation*}%
where the functions $f_{1},...,f_{2k}$ do not depend from $t.$ It follows
that 
\begin{equation*}
\left\Vert e^{tA}Te^{-tA}\right\Vert =O\left( \left( 1+\left\vert
t\right\vert \right) ^{2k}\right)
\end{equation*}%
and so $T\in \mathcal{D}_{A}^{2k}\left( 
\mathbb{R}
\right) .$ Thus we have $\mathcal{D}_{A}^{2k}\left( 
\mathbb{R}
\right) =B\left( X\right) .$
\end{proof}

As a consequence of Proposition 2.9, we have the following.

\begin{corollary}
Assume that $\Delta _{A}$ is a quasinilpotent for some $A\in B\left(
X\right) .$ If $K\left( X\right) \subset \mathcal{D}_{A}^{\alpha }\left( 
\mathbb{R}
\right) ,$ then $\Delta _{A}$ is a nilpotent of degree $\leq 2\left[ \alpha %
\right] +1.$
\end{corollary}

\begin{proof}
It follows from the identity 
\begin{equation*}
\sigma \left( \Delta _{A}\right) =\left\{ \lambda -\mu :\lambda ,\mu \in
\sigma \left( A\right) \right\} =\left\{ 0\right\}
\end{equation*}%
that $\sigma \left( A\right) $ consists of one point. By Proposition 2.9, $%
\mathcal{D}_{A}^{2\left[ \alpha \right] }\left( 
\mathbb{R}
\right) =B\left( X\right) .$ On the other hand, by Theorem 2.1, 
\begin{equation*}
\ker \Delta _{A}^{2\left[ \alpha \right] +1}=\mathcal{D}_{A}^{2\left[ \alpha %
\right] }\left( 
\mathbb{R}
\right) =B\left( X\right) .
\end{equation*}
\end{proof}

We conclude this section with the following result:

\begin{proposition}
The following assertions hold:

$a)$ For an arbitrary $T\in D_{A}\left( 
\mathbb{R}
\right) ,$ 
\begin{equation*}
\mathnormal{dist}\left( T,\left\{ A\right\} ^{\prime }\right) \leq
\sup_{t\in 
\mathbb{R}
}\left\Vert T-e^{tA}Te^{-tA}\right\Vert .
\end{equation*}

$b)$ If $D_{A}\left( 
\mathbb{R}
\right) $ is closed, then there exists a constant $C>0$ such that 
\begin{equation*}
\sup_{t\in 
\mathbb{R}
}\left\Vert T-e^{tA}Te^{-tA}\right\Vert \leq C\mathnormal{dist}\left(
T,\left\{ A\right\} ^{\prime }\right) ,\text{ \ }\forall T\in D_{A}\left( 
\mathbb{R}
\right) .
\end{equation*}
\end{proposition}

\begin{proof}
a) Let $T\in D_{A}\left( 
\mathbb{R}
\right) $ and $\delta :=\sup_{t\in 
\mathbb{R}
}\left\Vert T-e^{tA}Te^{-tA}\right\Vert .$ Define a mapping $\pi
:D_{A}\left( 
\mathbb{R}
\right) \rightarrow B\left( X\right) $ by 
\begin{equation*}
\langle \pi \left( T\right) x,\varphi \rangle =\Phi _{t}\langle
e^{tA}Te^{-tA}x,\varphi \rangle \text{ \ }\left( x\in X,\text{ }\varphi \in
X^{\ast }\right) ,
\end{equation*}%
where $\Phi $ is a fixed invariant mean on $%
\mathbb{Z}
.$ We claim that $\pi \left( T\right) \in \left\{ A\right\} ^{\prime }.$
Indeed, for an arbitrary $s\in 
\mathbb{R}
,$ from the identities 
\begin{eqnarray*}
\langle e^{sA}\pi \left( T\right) e^{-sA}x,\varphi \rangle &=&\langle \pi
\left( T\right) e^{-sA}x,e^{sA^{\ast }}\varphi \rangle \\
&=&\Phi _{t}\langle e^{tA}Te^{-tA}e^{-sA}x,e^{sA^{\ast }}\varphi \rangle \\
&=&\Phi _{t}\langle e^{\left( t+s\right) A}Te^{-\left( t+s\right)
A}x,\varphi \rangle \\
&=&\Phi _{t}\langle e^{tA}Te^{-tA}x,\varphi \rangle \\
&=&\langle \pi \left( T\right) x,\varphi \rangle ,
\end{eqnarray*}%
we have $e^{sA}\pi \left( T\right) =\pi \left( T\right) e^{sA}.$ This
clearly implies $A\pi \left( T\right) =\pi \left( T\right) A.$ Notice that $%
\langle \pi \left( T\right) x,\varphi \rangle $ belongs to the closure of
convex combination of the set 
\begin{equation*}
\left\{ \langle e^{tA}Te^{-tA}x,\varphi \rangle :t\in 
\mathbb{R}
\right\} .
\end{equation*}%
Now, from the inequality 
\begin{equation*}
\left\vert \langle Tx,\varphi \rangle -\langle e^{tA}Te^{-tA}x,\varphi
\rangle \right\vert \leq \delta \left\Vert x\right\Vert \left\Vert \varphi
\right\Vert
\end{equation*}%
we have 
\begin{equation*}
\left\vert \langle Tx,\varphi \rangle -\langle \pi \left( T\right) x,\varphi
\rangle \right\vert \leq \delta \left\Vert x\right\Vert \left\Vert \varphi
\right\Vert \text{, \ }\forall x\in X,\text{ }\forall \varphi \in X^{\ast }.
\end{equation*}%
Hence $\left\Vert T-\pi \left( T\right) \right\Vert \leq \delta .$ Since $%
\pi \left( T\right) \in \left\{ A\right\} ^{\prime }$, the result follows.

b) Applying uniform boundedness principle to the family of the operators 
\begin{equation*}
f_{t}:D_{A}\left( 
\mathbb{R}
\right) \rightarrow B\left( X\right) \text{; \ }f_{t}\left( T\right)
=e^{tA}Te^{-tA}\text{ }\left( t\in 
\mathbb{R}
\right) ,
\end{equation*}
we obtain existence of a constant $K>0$ such that 
\begin{equation*}
\left\Vert e^{tA}Te^{-tA}\right\Vert \leq K\left\Vert T\right\Vert \text{, \ 
}\forall T\in D_{A}\left( 
\mathbb{R}
\right) ,\text{ }\forall t\in 
\mathbb{R}
.
\end{equation*}%
Now if $T\in D_{A}\left( 
\mathbb{R}
\right) $ and $S\in \left\{ A\right\} ^{\prime }$, then as $T-S\in
D_{A}\left( 
\mathbb{R}
\right) ,$ from the relations 
\begin{eqnarray*}
\left\Vert T-e^{tA}Te^{-tA}\right\Vert &=&\left\Vert T-S-e^{tA}\left(
T-S\right) e^{-tA}\right\Vert \\
&\leq &\left( 1+K\right) \left\Vert T-S\right\Vert
\end{eqnarray*}%
we have 
\begin{equation*}
\sup_{t\in 
\mathbb{R}
}\left\Vert T-e^{tA}Te^{-tA}\right\Vert \leq \left( 1+K\right) \text{dist}%
\left( T,\left\{ A\right\} ^{\prime }\right) .
\end{equation*}
\end{proof}

\section{Decomposability}

Recall that an operator $A\in B\left( X\right) $ is called \textit{%
decomposable} if, for every open covering $\left\{ U,W\right\} $ of the
complex plane, there exist a pair of $T-$invariant closed linear subspaces $E
$ and $F$ of $X$ such that $E+F=X$, $\sigma \left( A\mid _{E}\right) \subset
U$ and $\sigma \left( A\mid _{F}\right) \subset W$ (for instance, see \cite%
{3} and \cite{9}). For a closed subset $F$ of $%
\mathbb{C}
,$ the \textit{local} \textit{spectral subspace} $X_{A}\left( F\right) $ of $%
A\in B\left( X\right) $ is defined by 
\begin{equation*}
X_{A}\left( F\right) =\left\{ x\in X:\sigma _{A}\left( x\right) \subseteq
F\right\} .
\end{equation*}%
If $A$ is decomposable, then $X_{A}\left( F\right) $ is closed for every
closed set $F\subset 
\mathbb{C}
$ \cite[Sect. 1.2]{9}. Moreover, 
\begin{equation*}
X_{f\left( A\right) }\left( F\right) =X_{A}\left( f^{-1}\left( F\right)
\right) ,
\end{equation*}%
for every closed set $F\subset 
\mathbb{C}
$, where $f\left( A\right) $ is defined by the Riesz-Dunford functional
calculus \cite[Theorem 3.3.6]{9}.

\begin{proposition}
Assume that $A\in B\left( X\right) $ is decomposable and $T\in B\left(
X\right) $ satisfies the condition 
\begin{equation*}
\left\Vert e^{nA}Te^{-nA}\right\Vert \leq C_{T}\left( 1+\left\vert
n\right\vert \right) ^{\alpha }\text{ }\left( \alpha \geq 0\right) ,
\end{equation*}%
for all $n\in 
\mathbb{Z}
$ and for some $C_{T}>0$. Then the following conditions are equivalent:

$\left( a\right) $ $TX_{A}\left( F\right) \subseteq X_{A}\left( F\right) $
for every closed set $F\subset 
\mathbb{C}
.$

$\left( b\right) \mathcal{\ }\Delta _{A}^{\left[ \alpha \right] +1}\left(
T\right) =0.$

In particular, if $0\leq \alpha <1,$ then $AT=TA$ if and only if $%
TX_{A}\left( F\right) \subseteq X_{A}\left( F\right) ,$ for every closed set 
$F\subset 
\mathbb{C}
.$
\end{proposition}

\begin{proof}
(a)$\Rightarrow $(b) We have%
\begin{equation*}
\left\Vert e^{n\Delta _{A}}\left( T\right) \right\Vert =\left\Vert
e^{nA}Te^{-nA}\right\Vert \leq C_{T}\left( 1+\left\vert n\right\vert \right)
^{\alpha }\text{, \ }\forall n\in 
\mathbb{Z}
.
\end{equation*}%
Since $A$ is decomposable, $\Delta _{A}$ has SVEP \cite[Proposition 3.4.6]{9}
and therefore, 
\begin{equation*}
r_{\Delta _{A}}\left( T\right) =\underset{n\rightarrow \infty }{\overline{%
\lim }}\left\Vert \Delta _{A}^{n}\left( T\right) \right\Vert ^{\frac{1}{n}}.
\end{equation*}%
On the other hand, $TX_{A}\left( F\right) \subseteq X_{A}\left( F\right) $
for every closed set $F\subset 
\mathbb{C}
$ if and only if 
\begin{equation*}
\lim_{n\rightarrow \infty }\left\Vert \Delta _{A}^{n}\left( T\right)
\right\Vert ^{\frac{1}{n}}=0
\end{equation*}
\cite[Corollary 3.4.5]{9}. Now, since $\sigma _{\Delta _{A}}\left( T\right)
=\left\{ 0\right\} $, by Corollary 2.8, $\Delta _{A}^{\left[ \alpha \right]
+1}\left( T\right) =0.$

In fact, (b)$\Rightarrow $(a) follows from \cite[Proposition 3.4.2]{9}.
Here, we present more simple proof. Now, it suffices to show that $\sigma
_{A}\left( Tx\right) \subseteq \sigma _{A}\left( x\right) $ for every $x\in
X.$ If $x\in X$ and $\lambda \in \rho _{A}\left( x\right) ,$ then there is a
neighborhood $U_{\lambda }$ of $\lambda $ with $u\left( z\right) $ analytic
on $U_{\lambda }$ having values in $X$, such that 
\begin{equation*}
\left( zI-A\right) u\left( z\right) =x\text{, \ }\forall z\in U_{\lambda }.
\end{equation*}%
Using this identity, it is easy to check that the function 
\begin{equation*}
v\left( z\right) :=Tu\left( z\right) -\Delta _{A}\left( T\right) u^{\prime
}\left( z\right) +...+\left( -1\right) ^{k}\Delta ^{k}\left( T\right)
u^{\left( k\right) }\left( z\right) \text{ \ }\left( k=\left[ \alpha \right]
\right) ,
\end{equation*}%
satisfies the equation 
\begin{equation*}
\left( zI-A\right) v\left( z\right) =Tx\text{, \ }\forall z\in U_{\lambda }.
\end{equation*}%
This shows that $\lambda \in \rho _{A}\left( Tx\right) .$
\end{proof}

Next, we have the following:

\begin{proposition}
Assume that the operators $A,T\in B\left( X\right) $ satisfy the following
conditions:

$\left( i\right) $ $A$ is decomposable and $\sigma \left( A\right) \subset
\left\{ z\in 
\mathbb{C}
:\func{Re}z>0\right\} ;$

$\left( ii\right) $ $\left\Vert A^{n}TA^{-n}\right\Vert \leq C_{T}\left(
1+\left\vert n\right\vert \right) ^{\alpha }$ $\left( 0\leq \alpha <1\right) 
$ for all $n\in 
\mathbb{Z}
$ and for some $C_{T}>0$.

Then, $AT=TA$ if and only if $TX_{A}\left( F\right) \subseteq X_{A}\left(
F\right) $ for every closed set $F\subset 
\mathbb{C}
$.
\end{proposition}

\begin{proof}
Assume that $TX_{A}\left( F\right) \subseteq X_{A}\left( F\right) $ for
every closed set $F\subset 
\mathbb{C}
$. We can write $A=e^{B},$ where $B=\log A.$ Then, $B$ is decomposable \cite[%
Theorem 3.3.6]{9} and%
\begin{equation*}
\left\Vert e^{nB}Te^{-nB}\right\Vert \leq C_{T}\left( 1+\left\vert
n\right\vert ^{\alpha }\right) \text{, \ }\forall n\in 
\mathbb{Z}
.
\end{equation*}%
Moreover, for every closed set $F\subset 
\mathbb{C}
,$ 
\begin{equation*}
TX_{B}\left( F\right) =TX_{A}(f^{-1}\left( F\right) \subseteq
X_{A}(f^{-1}\left( F\right) =X_{B}\left( F\right) ,
\end{equation*}%
where $f\left( z\right) =\log z.$ By Proposition 3.1, $BT=TB$ which implies $%
e^{B}T=Te^{B}.$ Hence $AT=TA.$

Assume that $AT=TA.$ It suffices to show that $\sigma _{A}\left( Tx\right)
\subseteq \sigma _{A}\left( x\right) $ for every $x\in X.$ If $x\in X$ and $%
\lambda \in \rho _{A}\left( x\right) ,$ then there is a neighborhood $%
U_{\lambda }$ of $\lambda $ with $u\left( z\right) $ analytic on $U_{\lambda
}$ having values in $X$, such that 
\begin{equation*}
\left( zI-A\right) u\left( z\right) =x\text{, \ }\forall z\in U_{\lambda }.
\end{equation*}%
It follows that 
\begin{equation*}
\left( zI-A\right) Tu\left( z\right) =Tx\text{, \ }\forall z\in U_{\lambda }.
\end{equation*}%
This shows that $\lambda \in \rho _{A}\left( Tx\right) .$
\end{proof}

\section{The norm of the commutator $AT-TA$}

In this section, we give some estimates for the norm of the commutator $%
AT-TA,$ where $T\in \mathcal{D}_{A}^{\alpha }\left( 
\mathbb{R}
\right) $ $\left( 0\leq \alpha <1\right) .$

\begin{lemma}
Let $\mu \in M_{\omega }\left( 
\mathbb{R}
\right) ,$ where $\omega \left( t\right) =\left( 1+\left\vert t\right\vert
\right) ^{\alpha }$ $\left( \alpha \geq 0\right) .$ Assume that $T\in
B\left( X\right) $ and $x\in X$ satisfy the following conditions:

$\left( i\right) $ $\left\Vert e^{tT}x\right\Vert \leq C\omega \left(
t\right) $ for all $t\in 
\mathbb{R}
$ and for some $C>0;$

$\left( ii\right) $ $T$ has SVEP.

If $\widehat{\mu }\left( \lambda \right) =\lambda $ in a neighborhood of $%
i\sigma _{T}\left( x\right) ,$ then 
\begin{equation*}
x_{\mu }=iTx.
\end{equation*}
\end{lemma}

\begin{proof}
Let $g\in S\left( 
\mathbb{R}
\right) $ be such that $\widehat{g}\left( \lambda \right) =1$ in a
neighborhood of $i\sigma _{T}\left( x\right) .$ By Proposition 2.5, $%
x_{g}=x. $ On the other hand, by Lemma 2.6, 
\begin{equation*}
x_{g^{\prime }}=-Tx.
\end{equation*}%
Since%
\begin{equation*}
\widehat{g^{\prime }}\left( \lambda \right) =i\lambda \widehat{g}\left(
\lambda \right) ,
\end{equation*}%
the Fourier transform of the function $-ig^{\prime }-\mu \ast g$ vanishes in
a neighborhood of $i\sigma _{T}\left( x\right) .$ By Proposition 2.5, 
\begin{equation*}
-ix_{g^{\prime }}=x_{\mu \ast g}=\left( x_{g}\right) _{\mu }=x_{\mu }.
\end{equation*}%
Hence $x_{\mu }=iTx.$
\end{proof}

Note that in the preceding lemma, the weight function $\omega \left(
t\right) =\left( 1+\left\vert t\right\vert \right) ^{\alpha }$ $\left(
\alpha \geq 0\right) $ can be replaced by the weight $\omega \left( t\right)
=1+\left\vert t\right\vert ^{\alpha }$ $.$

\begin{theorem}
Assume that $T\in B\left( X\right) $ has SVEP and $x\in X$ satisfies the
condition 
\begin{equation*}
\left\Vert e^{tT}x\right\Vert \leq C\left( 1+\left\vert t\right\vert
^{\alpha }\right) \text{ }\left( 0\leq \alpha <1\right) ,
\end{equation*}%
for all $t\in 
\mathbb{R}
$ and for some $C>0.$ Then we have%
\begin{equation*}
\left\Vert Tx\right\Vert \leq C\left[ r_{T}\left( x\right) +C\left( \alpha
\right) r_{T}\left( x\right) ^{1-\alpha }\right] ,
\end{equation*}%
where%
\begin{equation}
C\left( \alpha \right) =\left( \frac{2}{\pi }\right) ^{2-\alpha
}\sum\limits_{k\in 
\mathbb{Z}
}\frac{1}{\left\vert 2k+1\right\vert ^{2-\alpha }}.  \label{4.1}
\end{equation}
\end{theorem}

\begin{proof}
We basically follow the proof of Lemma 3.4 in \cite{12}. Let an arbitrary $%
a>r_{T}\left( x\right) $ be fixed. Consider the function $f,$ defined by $%
f\left( \lambda \right) =\lambda $ for $-a\leq \lambda \leq a$ and $f\left(
\lambda \right) =2a-\lambda $ for $a\leq \lambda \leq 3a.$ We extend this
function periodically to the real line by putting $f\left( \lambda
+4a\right) =f\left( \lambda \right) $ $\left( \lambda \in 
\mathbb{R}
\right) .$ A few lines of computation show that the Fourier coefficients of $%
f$ are given by the equalities: 
\begin{equation*}
c_{2k}\left( f\right) =0\text{, }c_{2k+1}\left( f\right) =\frac{1}{i}\frac{4a%
}{\pi ^{2}}\left( -1\right) ^{k}\frac{1}{\left( 2k+1\right) ^{2}}\text{ \ }%
\left( k\in 
\mathbb{Z}
\right) .
\end{equation*}%
Let $\mu $ be a discrete measure on $%
\mathbb{R}
$ concentrated at the points 
\begin{equation*}
\lambda _{k}:=-\frac{1}{a}\left( 2k+1\right) \frac{\pi }{2}\text{ \ }\left(
k\in 
\mathbb{Z}
\right) ,
\end{equation*}%
with the corresponding weights 
\begin{equation*}
c_{k}:=\frac{1}{i}\frac{4a}{\pi ^{2}}\left( -1\right) ^{k}\frac{1}{\left(
2k+1\right) ^{2}}\text{ \ }\left( k\in 
\mathbb{Z}
\right) .
\end{equation*}%
Since 
\begin{equation*}
\sum\limits_{k\in 
\mathbb{Z}
}\left\vert c_{k}\right\vert <\infty ,
\end{equation*}%
it follows from the uniqueness theorem that%
\begin{equation*}
\widehat{\mu }\left( \lambda \right) =f\left( \lambda \right) =\frac{1}{i}%
\frac{4a}{\pi ^{2}}\sum\limits_{k\in 
\mathbb{Z}
}\left( -1\right) ^{k}\frac{1}{\left( 2k+1\right) ^{2}}\exp \left[ i\frac{1}{%
a}\left( 2k+1\right) \frac{\pi }{2}\lambda \right] .
\end{equation*}%
Now, if $\omega \left( t\right) :=1+\left\vert t\right\vert ^{\alpha }$ $%
\left( 0\leq \alpha <1\right) ,$ then as 
\begin{equation*}
\sum\limits_{k\in 
\mathbb{Z}
}\frac{1}{\left( 2k+1\right) ^{2}}=\frac{\pi ^{2}}{4},
\end{equation*}%
we can write

\begin{eqnarray*}
\left\Vert \mu \right\Vert _{\omega } &=&\int\limits_{%
\mathbb{R}
}\left( 1+\left\vert t\right\vert ^{\alpha }\right) d\left\vert \mu
\right\vert \left( t\right) =\sum\limits_{k\in 
\mathbb{Z}
}\left\vert c_{k}\right\vert \left( 1+\left\vert \lambda _{k}\right\vert
^{\alpha }\right) \\
&=&\frac{4a}{\pi ^{2}}\sum\limits_{k\in 
\mathbb{Z}
}\frac{[1+\frac{1}{a^{\alpha }}\left\vert 2k+1\right\vert ^{\alpha }\left( 
\frac{\pi }{2}\right) ^{\alpha }]}{\left( 2k+1\right) ^{2}} \\
&=&a+C\left( \alpha \right) a^{1-\alpha },
\end{eqnarray*}%
where 
\begin{equation*}
C\left( \alpha \right) =\left( \frac{2}{\pi }\right) ^{2-\alpha
}\sum\limits_{k\in 
\mathbb{Z}
}\frac{1}{\left\vert 2k+1\right\vert ^{2-\alpha }}.
\end{equation*}%
Since $\widehat{\mu }\left( \lambda \right) =\lambda $ in a neighborhood of $%
i\sigma _{T}\left( x\right) ,$ by Lemma 4.1, $x_{\mu }=iTx$. Therefore, we
get 
\begin{equation*}
\left\Vert Tx\right\Vert =\left\Vert x_{\mu }\right\Vert \leq C\left\Vert
\mu \right\Vert _{\omega }\leq C\left[ a+C\left( \alpha \right) a^{1-\alpha }%
\right] .
\end{equation*}%
Since $a>r_{T}\left( x\right) $ is arbitrary, we obtain our result.
\end{proof}

As an application of Theorem 4.2, we have the following quantitative version
of Theorem 2.1 in the case $0\leq \alpha <1.$

\begin{corollary}
Let $A\in B\left( X\right) $ and assume that $\Delta _{A}$ has SVEP. If $%
T\in B\left( X\right) $ satisfies the condition 
\begin{equation*}
\left\Vert e^{tA}Te^{-tA}\right\Vert \leq C_{T}\left( 1+\left\vert
t\right\vert ^{\alpha }\right) \text{ }\left( 0\leq \alpha <1\right) ,
\end{equation*}%
for all $t\in 
\mathbb{R}
$ and for some $C_{T}>0,$ then 
\begin{equation*}
\left\Vert AT-TA\right\Vert \leq C_{T}\left[ r_{\Delta _{A}}\left( T\right)
+C\left( \alpha \right) r_{\Delta _{A}}\left( T\right) ^{1-\alpha }\right] ,
\end{equation*}%
where $C\left( \alpha \right) $ is defined by $\left( 4.1\right) .$
\end{corollary}

\begin{proof}
Noting that 
\begin{equation*}
\left\Vert e^{t\Delta _{A}}\left( T\right) \right\Vert =\left\Vert
e^{tA}Te^{-tA}\right\Vert \leq C_{T}\left( 1+\left\vert t\right\vert
^{\alpha }\right) \text{ \ }\left( \forall t\in 
\mathbb{R}
\right) ,
\end{equation*}%
by Theorem 4.2,%
\begin{equation*}
\left\Vert AT-TA\right\Vert \leq C_{T}\left[ r_{\Delta _{A}}\left( T\right)
+C\left( \alpha \right) r_{\Delta _{A}}\left( T\right) \right] .
\end{equation*}
\end{proof}


\begin{thebibliography}{99}
\bibitem{1} B.A. Barnes, Operators which satisfy polynomial growth
conditions, Pacific J. Math. 138(1989), 209-219.

\bibitem{2} C.J.K. Batty, J.M.A.M. van Neerven, and F. R\"{a}biger, Local
spectra and individual stability of uniformly bounded $C_{0}-$semigroups,
Trans. Amer. Math. Soc. 350(1998), 2087-2103.

\bibitem{3} I. Colojoar\u{a} and C. Foia\c{s}, Theory of Generalized
Spectral Operators, Gordon and Breach, New York, 1968.

\bibitem{4} J.A. Deddens, Another description of nest algebras in Hilbert
spaces operators, Lecture Notes in Math. 693(1978),\textbf{\ }77-86.

\bibitem{5} N. Dunford and J.T. Schwartz, Linear Operators III, New York,
Wiley-Interscience, 1971.

\bibitem{6} I. Gelfand, D. Raikov and G. Shilov, Commutative Normed Rings,
Chelsea Publ. Company, New York, 1964.

\bibitem{7} V.P. Gurarii,\textit{\ }Harmonic analysis in spaces with weight,
Trans. Moscow Math. Soc. 35(1979), 21-75.

\bibitem{8} R. Larsen, Banach Algebras, Marcel Dekker, New York, 1973.

\bibitem{9} K.B. Laursen and M.M. Neumann,\textit{\ }An Introduction to
Local Spectral Theory, Oxford, Clarendon Press, 2000.

\bibitem{10} R.I. Loebl and P.S. Muhly, Analyticity and flows in von Neumann
algebras, J. Funct. Anal. 29(1978), 214-252.

\bibitem{11} Yu.I. Lyubich, Introduction to the Theory of Banach
Representation of Groups, Oper. Theory, Adv. Appl. vol.30, Birkh\"{a}user,
1988.

\bibitem{12} H. Mustafayev, Dissipative operators on Banach spaces, J.
Funct. Anal. 248(2007), 428-447.

\bibitem{13} H. Mustafayev, Growth conditions for operators with smallest
spectrum, Glasg. Math. J. 57(2015), 665-680.

\bibitem{14} J. van Neerven, The Asymptotic Behaviour of Semigroups of
Linear Operators, Oper. Theory Adv. Appl. vol.88, Birkh\"{a}user,\ 1996.

\bibitem{15} H. Reiter,\textit{\ }Classical Harmonic Analysis and Locally
Compact Groups, Oxford Univ. Press, 1968.

\bibitem{16} J. Ringrose, On some algebras of operators, Proc. London Math.
Soc. 15(1965), 61-83

\bibitem{17} P.G. Roth,\textit{\ }Bounded orbits of conjugation, analytic
theory, Indiana Univ. Math. J. 32(1983), 491-509.

\bibitem{18} J.P. Williams, On a boundedness condition for operators with a
singleton spectrum, Proc. Amer. Math. Soc. 78(1980), 30-32.

\bibitem{19} M. Zarrabi, Spectral synthesis and applications to $C_{0}-$%
groups, J. Austral. Math. Soc. (Series A), 60(1996), 128-142.
\end{thebibliography}
\end{document}